\documentclass[10pt]{article}
\usepackage{latexsym,amsmath,amsthm,verbatim,ifthen,amssymb}
\usepackage[latin1]{inputenc}
\usepackage[all]{xy}
\usepackage{graphicx}
\usepackage{epsfig}
\newdir{ >}{{}*!/-7pt/\dir{>}}

\newtheorem{theorem}{Theorem}
\newtheorem{corollary}[theorem]{Corollary}
\newtheorem{lemma}[theorem]{Lemma}
\newtheorem{proposition}[theorem]{Proposition}

\newtheorem{definition}[theorem]{Definition}

\def\nil{{\rm{nil}\hskip1pt}}
\def\cat{{\rm{cat}\hskip1pt}}
\def\wcat{{\rm{wcat}\hskip1pt}}
\def\secat{{\rm{secat}\hskip1pt}}
\def\wsecat{{\rm{wsecat}\hskip1pt}}
\def\TC{{\rm{TC}\hskip1pt}}
\def\wTC{{\rm{wTC}\hskip1pt}}
\def\cuplength{{\rm{cuplength}\hskip1pt}}

\def\CP{{\mathbb{C}\mathbb{P}}}

\begin{document}

\title{Weak sectional category\footnotetext{This work has been supported
 by the Ministerio de Educaci\'on y Ciencia
grant MTM2006-06317, the Portuguese Foundation for Science and Technology (FCT/POCI 2010) and FEDER.}}
\author{J.M. Garc\'{\i}a Calcines\footnote{Universidad de La Laguna,
Facultad de Matem\'aticas, Departamento de Matem\'atica
Fundamental, 38271 La Laguna, Spain. E-mail:
\texttt{jmgarcal@ull.es}} \, and L. Vandembroucq\footnote{Centro
de Matem{\'a}tica, Universidade do Minho, Campus de Gualtar,
4710-057 Braga, Portugal. E-mail: \texttt{lucile@math.uminho.pt}}
}

\date{\empty}

\maketitle

\begin{abstract}
Based on a Whitehead-type characterization of the sectional
category we develop the notion of
weak sectional category. This is a new lower bound of the
sectional category, which is inspired by the notion of weak
category in the sense of Berstein-Hilton. We establish several
properties and inequalities, including the fact that the
weak sectional category is a better lower bound for the sectional category than the classical
one given by the nilpotency of the kernel of the induced map in
cohomology. Finally, we apply our results in the study of the
topological complexity in the sense of Farber.
\end{abstract}

\vspace{0.5cm}
\vspace{0.2cm}

\section*{Introduction.}
For a fibration $p:E\rightarrow B,$ the sectional category $\secat
(p)$ is defined as the least integer $k$ such that $B$ admits a
cover constituted by $k+1$ open subsets, on each of which $p$ has a
 section. Requiring homotopy section instead of section permits
 to extend the notion of sectional category to any map. This is
 a variant of Lusternik-Schnirelmann category (or L-S
category, for short) and also a generalization, since $\secat
(p)=\cat (B)$ when $E\simeq *.$  The sectional category has been
introduced for fibrations by A. Schwarz \cite{Sch} in the 1960's,
under the name \textit{genus} (then it was renamed by I.M. James
\cite{J}). It has turned out to be a useful tool not only in
questions concerning the classification of bundles, the embedding
problem (see \cite{Sch} for both applications) or the complexity
of the root-finding problem for algebraic equations \cite{Sm}, but
also, more recently, in the study of the motion planning problem
in robotics \cite{Far}, \cite{Far2}.

In the development of L-S category, Ganea's and Whitehead's
characterizations have played a very important role. These
characterizations are not only easier to handle in the framework
of homotopy theory than the open set definition but they also
permit to obtain many approximations of L-S category. Since one of
the most successful aproximations of L-S category is the weak
category in the sense of Berstein-Hilton \cite{B-H}, our aim in
this paper is to develop an analogous lower bound for sectional
category, the \textit{weak sectional category}. Taking into
account that the classical weak category is based on Whitehead's
characterization we need a Whitehead-type formulation of sectional
category. Such a characterization has been obtained by A. Fassò
\cite{F}. More precisely, for a cofibration $f:A\to X$ and an
integer $n\geq 0$, she defined the $n$-fatwedge of $f$ as the
subspace of $X^{n+1}$ given by $T^n(f)=\{(x_0,...,x_n)\in
X^{n+1}\, |\,x_i\in A,\hspace{3pt} \mbox{for some}\hspace{3pt}i\}$
and she proves that, for a map $p:E\to B$ with associated
cofibration $\hat p:E\to \hat B$, $\secat(p)\leq n$ if and only if
the diagonal map $\Delta_{n+1}:\hat B\to \hat B^{n+1}$ lifts, up
to homotopy, in the $n$-fatwegde of the cofibration $\hat p$:
$$\xymatrix{
{} & {T^n(\hat p)} \ar@{>->}[d] \\
{\hat B} \ar@{.>}[ur] \ar[r]_{\Delta _{n+1}} & {\hat B^{n+1}} }$$
In section 1 we use a notion of fatwedge for any map which
generalizes Fassò's fatwedge and recover her result. Namely, for
any map $p:E\to B$, we construct what we call the n-sectional
fatwedge of $p$, $\kappa _n:T^n(p)\rightarrow B^{n+1},$ and prove
that $\secat(p)\leq n$ if and only if $\Delta_{n+1}:B\to B^{n+1}$
lifts, up to homotopy, in $T^n(p)$.

\medskip
The notion of weak sectional category comes naturally. Considering
$C_{\kappa_n}$ the homotopy cofibre of the $n$-sectional fat wedge
of $p$, $\kappa _n:T^n(p)\rightarrow B^{n+1}$, and considering
$l_n:B^{n+1}\rightarrow C_{\kappa_n}$ the induced map, the weak
sectional category of $p,$ $\wsecat(p),$ is the least integer $n$
(or $\infty$) such that the composition $$\xymatrix{
B\ar[r]^{\Delta_{n+1}}& B^{n+1}\ar[r]^{l_n}& C_{\kappa_n}}$$ is
homotopically trivial. Section 2 is fully devoted to the study of
this new lower bound of sectional category. This study can be
summarized in the following theorem (see section 2 for more
details)

\medskip
\textbf{Theorem.} Let $p: E\to B$ be a map and $C_p$ be its
homotopy cofibre. Then
\begin{enumerate}
\item[(a)] $\wsecat(p)\leq \wcat(B)$
\item[(b)] $\wsecat(p)\leq \wcat(C_p)$
\item[(c)] $\wsecat(p)\geq \wcat(C_p)-1$
\item[(d)] $\wsecat(p)\geq \nil \ker p^*$ (for any commutative ring $\pi$)
\item[(e)] If the map $p: E\to B$ admits a homotopy retraction, then
$$\wsecat(p)=\wcat(C_p)\hspace{5pt}\mbox{and}\hspace{5pt}\nil \ker p^*=\cuplength(C_p).$$
\end{enumerate}
Here $\nil \ker p^*$ denotes the nilpotency of the kernel of the morphism $p^*$ which is induced by $p$ in
cohomology. This is a classical lower bound for sectional category \cite{Sch}. By inequality (d) above, weak sectional category turns out to be a better lower bound than $\nil \ker p^*$ and we will see that actually the inequality can be strict.
In Section 2 we also give sufficient conditions to assure that
$\wsecat(p)=\secat (p).$

\medskip
On the other hand, the motion planning problem is an important
research area in robotics in which topological methods can be
applied. The motion planning problem of the configuration space
$X$ (associated to a mechanical system) consists of constructing a
program or a devise, which takes pairs of configurations $(a,b)\in
X\times X$ as an input and produces as an output a continuous path
$\omega $ in $X$ such that $\omega (0)=a$ and $\omega (1)=b.$
Studying such a research area M. Farber introduced in \cite{Far},
\cite{Far2} the notion of topological complexity, which is a
numerical homotopical invariant of the space $X.$ It is defined as
the sectional category of the evaluation fibration $\pi
_X:X^I\rightarrow X\times X,$ $\pi _X(\alpha )=(\alpha (0),\alpha
(1)).$  In the last section of this paper we study what we call
\textit{weak topological complexity}, $\wTC(X),$ which is nothing
else but the weak sectional category of the evaluation fibration
$\pi _X.$ This is a new lower bound of the topological complexity.
Specializing the results given in section 2 we obtain the
corresponding ones, summarized as

\medskip
\textbf{Theorem.} Let $X$ be any space. Then
\begin{enumerate}
\item[(a)] $\wTC(X)\leq \wcat(X\times X)$
\item[(b)] $\wTC(X)\geq \nil \ker \hspace{3pt}\cup$
\item[(c)] $\wTC(X)=\wcat(C_{\Delta _X})$ (and $\nil \ker \hspace{3pt}\cup
=\cuplength(C_{\Delta _X})$).
\end{enumerate}

We also give sufficient conditions to assure that
$\wTC(X)=\TC(X).$ Finally we give two concrete computations of
$\wTC$. The first one consists of the explicit determination of
the homotopy cofibre of the diagonal map $\Delta_{S^n}: S^n\to
S^n\times S^n$ of a sphere. This fact permits us to recover the
result by Farber that the topological complexity of an odd
dimensional sphere is $1$ while that of an even dimensional sphere
is $2$. The second computation has the aim to prove, through an
example, that $\wTC$ is, in general, a better lower bound for the
topological complexity than
$\nil \ker \hspace{3pt}\cup$.\\

Throughout this paper we work with homotopy commutative diagrams,
homotopy pullbacks, homotopy pushouts and joins as well as some of
their more important properties. We assume that the reader is
familiarized with this framework; nevertheless at the very
beginning of Section 1 we recall the notion of join, the Join
Theorem and the Prism Lemma. For more details we refer the reader
to \cite{M}, \cite {D} and \cite{B-K}. We also point out that the
category in which we shall work is the category of well-pointed
compactly generated Hausdorff spaces. Therefore all categorical
constructions are carried out in this category.

\section{Joins and characterizations of sectional category.}

The main goal of this section is to give some important notions
and results that will be needed in the rest of the paper. We begin
by recalling the notion of join of two maps. Given any pair of
maps $A\stackrel{f}{\longrightarrow}C\stackrel{g}{\longleftarrow
}B$, the \textit{join of} $f$ \textit{and} $g,$ $A*_C B,$ is the
homotopy pushout of the homotopy pullback of $f$ and $g$

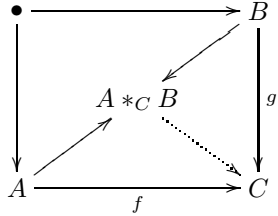
$$\xymatrix@C=0.7cm@R=0.7cm{ {\bullet } \ar[rr] \ar[dd] & & {B} \ar[dl] \ar[dd]^g \\
 & {A*_C B} \ar@{.>}[dr] & \\ {A} \ar[ur] \ar[rr]_f & & {C} }$$
being the dotted arrow the corresponding co-whisker map, induced
by the weak universal property of homotopy pushouts.

Notice that the map $A*_C B\to C$ is only defined up to weak
equivalence.  Any map constructed in such a way is weakly
equivalent to the canonical co-whisker map obtained by considering
first the standard homotopy pullback of $f$ and $g$ and then the
standard homotopy pushout of the projections on $A$ and $B$. In
order to give another example of a concrete construction of the
join, suppose that $f$ (or $g$) is a fibration. Then the honest
pullback of $f$ and $g$ is actually a homotopy pullback and taking
the standard homotopy pushout of the projections $A\times_C B\to
A$ and $A\times_C B\to B$ we obtain a representative of the map
$A*_C B\to C$ of the form:
$$\begin{array}{rcl}
A\ast_C B=A\amalg A\times_C B\times [0,1] \amalg B/\sim &\to&C\\
&&\\
\langle a,b,t\rangle & \mapsto & f(a)=g(b)\\
\end{array}$$
where $\sim$ is given by $(a,b,t)\sim a$ if $t=0$ and $(a,b,t)\sim b$ if $t=1$.\\

In \cite{D}, Doeraene
established the following result, called the ``Join Theorem".

\begin{theorem}[\textbf{Join Theorem}]
Consider the homotopy commutative diagram in which the squares are
homotopy pullbacks
$$\xymatrix{
{A} \ar[r]^f \ar[d] & {C} \ar[d]^{\gamma } & {B} \ar[l]_g \ar[d]
\\ {A'} \ar[r]_{f'} & {C'} & {B'} \ar[l]^{g'}  }$$
Then there is a homotopy pullback
$$\xymatrix{
{A*_C B} \ar[r] \ar[d] & {C} \ar[d]^{\gamma } \\
{A' *_{C'}B'} \ar[r] & {C'} }$$ \noindent where the horizontal
arrows are the corresponding induced co-whisker maps.
\end{theorem}

We also mention the following classical result that we will use in
the sequel, as well as its dual version for homotopy pushouts.

\begin{proposition}[\textbf{Prism Lemma}] Given any homotopy commutative diagram,
$$\xymatrix{
  A \ar[dd] \ar[rr] \ar[dr] & &    C \ar[dd] \\
  & B  \ar[dd] \ar[ur] & \\
  D \ar[dr] \ar'[r][rr] & & F \\
   &  E  \ar[ur] &  }$$
\noindent where the square $BCFE$ is a homotopy pullback. Then,
$ACFD$ is a homotopy pullback if and only $ABED$ is a homotopy
pullback.
\end{proposition}

The basic notion of interest in this work is the sectional category:

\begin{definition}
The sectional category of a map $p:E\rightarrow B,$ \secat$(p)$, is the
least integer $n$ (or $\infty$) such that $B$ can be covered by $n+1$ open
subsets, on each of which $p$ admits a homotopy section.
\end{definition}

There is a classical Ganea-type characterization of this homotopy
numerical invariant which is as follows.  Given any map
$p:E\rightarrow B$ we may consider the join map $E*_B E\rightarrow
B$ resulting from the join of $p$ with $p.$ This construction can
be done again, considering the join of this map with $p.$ In this
way we inductively obtain  $j^n_p:*^n_B E\rightarrow B$ as the
join of $j^{n-1}_p$ and $p$ (we set $j^0_p=p,$ $*^0_B E=E$). The
characterization of $\secat$ is then given by the following
classical result, see for instance \cite{J}.

\begin{theorem}
Let $p:E\rightarrow B$ be a map. If $B$ is a paracompact space,
then \secat$(p)\leq n$ if and only if $j^n_p$ admits a homotopy
section.
\end{theorem}
We remark that there is an improvement of the above theorem in
\cite{F-G-K-V}, in which $B$ is required to be just a normal
space.

As said in the introduction, A. Fassò gave in \cite{F} the
following Whitehead-type characterization of sectional category:

\begin{theorem}
Let $p:E\rightarrow B$ be a map and let $\hat p:E\to \hat B$ be
the associated  cofibration of $p$. Then $\secat(p)\leq n$ if and
only if the diagonal map $\Delta_{n+1}:\hat B\to \hat B^{n+1}$
lifts, up to homotopy, in the $n$-fatwegde $T^n(\hat
p)=\{(b_0,...,b_n)\in \hat B^{n+1}\, |\,b_i\in E,\hspace{3pt}
\mbox{for some}\hspace{3pt}i\}$ of the cofibration $\hat p$:
$$\xymatrix{
{} & {T^n(\hat p)} \ar@{>->}[d] \\
{\hat B} \ar@{.>}[ur] \ar[r]_{\Delta _{n+1}} & {\hat B^{n+1}} }$$

\end{theorem}

We give an alternative proof of this result based on a notion of
fatwedge defined for any map and called \textit{sectional
fatwedge}. We note that Doeraene \cite{D2} used this construction
for different purposes as an auxiliary notion under the name of
relative fatwedge.

\begin{definition}
{\rm The $n$-\textit{sectional fatwedge} of any map
$p:E\rightarrow B$
$$\kappa _n:T^n(p)\rightarrow B^{n+1}$$ \noindent is defined inductively as follows:
\begin{enumerate}
\item[(i)]
For $n=0$ we set $\kappa _0=p:E\rightarrow B$

\item[(ii)] Having defined $\kappa _{n-1}:T^{n-1}(p)\rightarrow B^n$
then $\kappa_n$ is obtained by considering the join of $\kappa
_{n-1}\times id_B$ and $id_{B^n}\times p$:
$$\xymatrix@C=0.5cm@R=0.6cm{ {\bullet } \ar[rr] \ar[dd] & & {B^n\times E}
\ar[dl] \ar[dd]^{id_{B^n}\times p} \\
 & {T^n(p)} \ar@{.>}[dr]^{\kappa _n} & \\
 {T^{n-1}(p)\times B} \ar[ur] \ar[rr]_{\kappa _{n-1}\times id_B} & & {B^{n+1}} }$$
\end{enumerate} }
\end{definition}

We will use the following lemma. Its proof is straightforward
and therefore is omitted.

\begin{lemma}
{\rm Consider $f:X\rightarrow Y$ any map and $Z$ any space. Then,
the following square, where the
horizontal arrows are the obvious projections, is a pullback and a homotopy pullback
$$\xymatrix{
{X\times Z} \ar[r]^(.6){pr} \ar[d]_{f\times id_Z} & {X} \ar[d]^f \\
{Y\times Z} \ar[r]_(.6){pr} & {Y.} }$$ }
\end{lemma}

\begin{theorem}
Let $p:E\rightarrow B$ be any map and $n\geq 0$ any nonnegative
integer. Then there is a homotopy pullback
$$\xymatrix{
{*^n_B E} \ar[d]_{j^n_p} \ar[rr]^{\widetilde{\Delta }_{n+1}} & &
{T^n(p)} \ar[d]^{\kappa _n} \\ {B} \ar[rr]_{\Delta _{n+1}} & &
{B^{n+1}} }$$
\end{theorem}

\begin{proof}
We proceed by induction on the integer $n\geq 0.$ For $n=0$ the
result is obviously true. Now suppose the result is true for
$n-1.$ Considering the Prism Lemma, the above lemma and the
induction hypothesis to the following homotopy commutative diagram
$$\xymatrix{
{*^{n-1}_B E} \ar[d]^{j^{n-1}_p} \ar[rr]^{(\widetilde{\Delta }
_n,j^{n-1}_p)} & & {T^{n-1}(p)\times B} \ar[d]^{\kappa
_{n-1}\times
id_B} \ar[rr]^{pr} & & {T^{n-1}(p)} \ar[d]^{\kappa _{n-1}} \\
{B} \ar[rr]_{\Delta _{n+1}} & & {B^{n+1}} \ar[rr]_{pr} &  & {B^n}
}$$ \noindent we have that the left square is a homotopy pullback.
Now applying a similar argument to the homotopy commutative
diagram
$$\xymatrix{
{E} \ar[d]^{p} \ar[rr]^(.4){(p,...,p,id_E)} & & {B^n\times E}
\ar[d]^{id_{B^n}\times p} \ar[rr]^{pr} & & {E} \ar[d]^{p} \\
{B} \ar[rr]_{\Delta _{n+1}} & & {B^{n+1}} \ar[rr]_{pr} &  & {B}
}$$ \noindent we have that the left square is also a homotopy
pullback. Summarizing, we have obtained a homotopy commutative
diagram in which the squares are homotopy pullbacks
$$\xymatrix{
{E} \ar[rr]^{p} \ar[d]^{(p,...,p,id_E)} & & {B} \ar[d]^{\Delta
_{n+1}}  & & {*^{n-1}_BE} \ar[ll]_{j^{n-1}_p}
\ar[d]^{(\widetilde{\Delta } _n,j^{n-1}_p)} \\
{B^n\times E} \ar[rr]_{id_{B^n}\times p} & & {B^{n+1}} & &
{T^{n-1}(p)\times B} \ar[ll]^{\kappa _{n-1}\times id_B} }$$
Applying the Join Theorem to this diagram we conclude the
proof.
\end{proof}

\begin{corollary}
\label{characterization-secat}
{\rm Let $p:E\rightarrow B$ be any map, where $B$ is a normal
space. Then one has \secat$(p)\leq n$ if and only if there is, up to
homotopy, a lift of the $(n+1)$-diagonal map }
$$\xymatrix{
{} & {T^n(p)} \ar[d]^{\kappa _n} \\
{B} \ar@{.>}[ur] \ar[r]_{\Delta _{n+1}} & {B^{n+1}} }$$
\end{corollary}

Now we will see that $T^n(p)$ coincides with Fassò's fatwedge when
$p$ is a cofibration.

\begin{lemma}
Let $p:E\rightarrow B$ be any map. Then the following square is a
homotopy pullback
$$
\xymatrix{ {T^{n-1}(p)\times E} \ar[rr]^(.55){\kappa _{n-1}\times
id_E} \ar[d]_{id\times p} & & {B^n\times E}
\ar[d]^{id_{B^n}\times p} \\
{T^{n-1}(p)\times B} \ar[rr]_(.55){\kappa _{n-1}\times id_B} & &
{B^{n+1}} }
$$
\end{lemma}

\begin{proof}
Take the standard homotopy pullback of $\kappa _{n-1}\times id_B$
and $id_{B^n}\times p,$ that is
$$\begin{array}{lcr}
L&=&\{(a,x,\gamma ,b,y)\in T^{n-1}(p)\times B\times
(B^{n+1})^I\times B^n\times E \mbox{ such that }\quad \quad\\
  && \gamma (0)=(\kappa _{n-1}(a),x),\hspace{3pt}\gamma
(1)=(b,p(y))\}
\end{array}$$
Consider the maps $$\xi :L\rightarrow
T^{n-1}(p)\times E\hspace{5pt}\mbox{and}\hspace{5pt}\pi
:T^{n-1}(p)\times E\rightarrow L$$ \noindent given by
$\xi(a,x,\gamma ,b,y)=(a,y)$ and $\pi (a,y)=(a,p(y),\varepsilon
_{(\kappa _{n-1}(a),p(y))},\kappa _{n-1}(a),y)$, where
$\varepsilon _{(\kappa _{n-1}(a),p(y))}$ denotes the constant path
on $(\kappa _{n-1}(a),p(y)).$
Then $\xi \pi =id$ and $\pi \xi \simeq id$ through the homotopy
$F:L\times I\rightarrow L$ defined as
$$F(a,x,\gamma ,b,y;t)=(a,\gamma _2(t),\delta (t),\gamma
_1(1-t),y)$$ \noindent where $\gamma =(\gamma _1,\gamma _2)$ with
$\gamma _1:I\rightarrow B^n,$ $\gamma _2:I\rightarrow B$ and
$$\delta (t)(s)=(\gamma _1(s(1-t)),\gamma _2((1-s)t+s))$$
This proves that $\pi$ is a homotopy equivalence and therefore the
statement of the lemma.
\end{proof}

Taking into account this fact we have that $T^n(p)$ may be
directly obtained as the homotopy pushout of $\kappa _{n-1}\times
id_E$ and $id\times p.$ In particular when $p$ is a cofibration we
can consider the honest pushout. Using induction we have:

\begin{corollary}
\label{fatwedge-cofibration}
Let $p:E\rightarrowtail B$ be a cofibration. Then
$$T^n(p)=\{(b_0,b_1,...,b_n)\in B^{n+1}\, |\,b_i\in E,\hspace{3pt}
\mbox{for some}\hspace{3pt}i\}$$ \noindent and the natural
inclusion $\kappa _{n}:T^n(p)\rightarrowtail B^{n+1}$ is a
cofibration.
\end{corollary}

Rigorously, this is an equality up to homotopy, since $T^n(p)$ is
only defined up to homotopy. However, in the sequel, when $p$ is
cofibration, the notation $T^n(p)$ will denote exactly this
subspace of $B^{n+1}$. Notice that if $E=*,$ then $T^n(p)$ is the
usual $n$-fatwedge of $B,$ given by $T^n(B)=\{(b_0,b_1,...,b_n)\in
B^{n+1}\, |\,b_i=*,\hspace{3pt} \mbox{for some}\hspace{3pt}i\}$.
Also observe that if $p:E\rightarrow B$ is any map, then we can
factorize $p$ as a cofibration $\hat{p}:E\rightarrowtail
\widehat{B}$ followed by a homotopy equivalence
$h:\widehat{B}\stackrel{\simeq }{\rightarrow }B$. In this case, we
have $\mbox{secat}(p)=\mbox{secat}(\hat{p})$ and $T^n(p),$
$T^n(\hat{p})$ are related by a homotopy commutative diagram
$$\xymatrix{
{T^n(\hat{p})} \ar[r]_{\simeq } \ar@{ >->}[d] & {T^n(p)} \ar[d] \\
{(\widehat{B})^{n+1}} \ar[r]_{h^{n+1}}^{\simeq } & {B^{n+1}} }$$
\noindent where the horizontal arrows are homotopy equivalences.
This last observation together with Corollaries
\ref{characterization-secat} and \ref{fatwedge-cofibration} shows
that our characterization of sectional category given in Corollary
\ref{characterization-secat} is equivalent to the one given by
Fassò in \cite{F}.

We end this section with a useful property of the sectional fatwedge of a cofibration.

\begin{proposition}\label{hp-sectional-fatwedge}
Given any pushout of the form
$$\xymatrix{
{E} \ar@{ >->}[d]_p \ar[r] & {E'} \ar@{ >->}[d]^{p'} \\
{B} \ar[r] & {B'} }$$ \noindent then, for any $n\geq 0,$ there
exists a pushout
$$\xymatrix{
{T^n(p)} \ar@{ >->}[d]_{\kappa _n} \ar[r] & {T^n(p')} \ar@{ >->}[d]^{\kappa '_n} \\
{B^{n+1}} \ar[r] & {(B')^{n+1}} }$$
\end{proposition}

We remark that this result can also be stated in terms of general
homotopy pushouts, in which $p$ and $p'$ are not necessarily
cofibrations. In the form above, it follows directly from an
inductive argument and the lemma below which is certainly
well-known so that we just give the idea of its proof.

\begin{lemma}\label{product_of_pushouts}
Consider the following honest pushout diagrams, in which $f$ and
$f'$ are cofibrations:
$$\xymatrix{
A\ar[r]^{g} \ar@{ >->}[d]_ f & Y\ar@{ >->}[d]^{\bar f} &\qquad &
A'\ar[r]^{g'} \ar@{ >->}[d]_{f'} & Y'\ar@{ >->}[d]^{\bar f'} \\
X \ar[r]_{\bar g} & Z &\qquad & X' \ar[r]_{\bar g'} & Z'. }$$ Then
the cofibrations $(A\times X')\cup_{A\times A'} (X\times A')\to
X\times X'$ and $(Y\times Z')\cup_{Y\times Y'} (Z\times Y')\to
Z\times Z',$ induced respectively by $f\times id$ and $id\times
f'$ and by $\bar f\times id$ and $id\times \bar f',$ fit in the
following commutative square (with the obvious maps). Furthermore,
this square is a pushout:
$$\xymatrix{
(A\times X')\cup_{A\times A'} (X\times A')\ar[r] \ar@{ >->}[d] & (Y\times Z')\cup_{Y\times Y'} (Z\times Y')\ar@{ >->}[d]  \\
X\times X' \ar[r] & Z\times Z'. }$$
\end{lemma}

\begin{proof}
The fact that the diagram is a pushout follows from the unicity of
the colimit of the following commutative diagram:
$$\xymatrix{
Y\times Z' & A\times Y'\ar@{ >->}[r]\ar[l] & X\times Y'\\
A\times X' \ar[u] & A\times A' \ar[u] \ar@{ >->}[d]\ar@{ >->}[r]\ar@{ >->}[l] & X\times A'\ar@{ >->}[d]\ar[u]   \\
A\times X' \ar@{=}[u] & A\times X' \ar@{=}[l]  \ar@{ >->}[r] & X\times X'.
}$$
\end{proof}

\section{Weak sectional category.}

For any $n\geq 0$ we shall denote by $B^{[n+1]}$ the $(n+1)$-fold
smash-product and by $\bar \Delta_{n+1}=\bar \Delta^B_{n+1}:B\to
B^{n+1}\to B^{[n+1]}$ the reduced diagonal. Recall that the weak
category of $B$, $\wcat(B)$, introduced by Berstein and Hilton
\cite{B-H}, is the least integer $n$ such that $\bar \Delta_{n+1}$
is homotopically trivial. This is a lower bound for the L-S
category of $B$ since $\cat(B)\leq n$ if and only if the diagonal
$\Delta^B_{n+1}:B\to B^{n+1}$ lifts, up to homotopy, in the
fatwedge $T^n(B)=\{(b_0,b_1,...,b_n)\in B^{n+1}:b_i=*,\hspace{3pt}
\mbox{for some}\hspace{3pt}i\}$ and $B^{[n+1]}=B^{n+1}/T^n(B)$.

Following the same idea, we introduce here the weak sectional
category of a map $p,$ which will be a lower bound for the
sectional category and a generalization of the weak category of
Berstein and Hilton. This section is then devoted to the study of
this invariant. In particular, as $\wcat(B)$ is, in general, a
better lower bound for $\cat(B)$ than the cuplength of $B$ (recall
that, for any commutative ring $\pi$, $\cuplength(B)$ is the
nilpotency of the reduced cohomology $\tilde H^*(B;\pi)$, that is
the least integer $n$ such that all $(n+1)$-fold cup products
vanish in $\tilde H^*(B;\pi)$), we will show that the weak
sectional category of $p$ is a better lower bound for $\secat(p)$
than the classical cohomological lower bound $\nil \ker p^*$ given
by the nilpotency of the ideal $\ker p^*\subset \tilde
H^*(B;\pi)$.

\subsection{Definition and a characterization}

\begin{definition}
Let $p: E\to B$ be a map and, for any integer $n$, let
$C_{\kappa_n}$ be the homotopy cofibre of the $n$-sectional fat
wedge of $p$, $\kappa _n:T^n(p)\rightarrow B^{n+1}$. If
$l_n:B^{n+1}\rightarrow C_{\kappa_n}$ denotes the induced map,
then the weak sectional category of $p$, denoted by $\wsecat(p)$,
is the least integer $n$ (or $\infty$) such that the composition
$$\xymatrix{ B\ar[r]^{\Delta_{n+1}}& B^{n+1}\ar[r]^{l_n}&
C_{\kappa_n}}$$ is homotopically trivial.

\end{definition}

It follows directly from the definition that, if $B$ is normal,
then $\wsecat(p)\leq \secat(p)$ and that, if $E=*$, then
$\wsecat(p:*\to B)=\wcat(B)$. We also observe that if $p=h\circ
\hat{p},$ where $\hat{p}$ is a cofibration and $h$ is a homotopy
equivalence, then $\wsecat(p)=\wsecat(\hat p)$.\\

The next result is a characterization of $\wsecat(p)$ in terms of
the homotopy cofibre of $p$ which will be useful in order to
establish the properties of the weak sectional category that will
be seen in the sequel.

\begin{proposition}\label{characterization_wsecat}
Let $p: E\to B$ be a map, $C_p$ be its homotopy cofibre, and
$j:B\to C_p$ be the induced map. Then $\wsecat(p)\leq n$ if and
only if the map $\xymatrix{ B\ar[r]^{\bar \Delta_{n+1}}&
B^{[n+1]}\ar[r]^{j^{[n+1]}}&C_p^{[n+1]}}$ is homotopically trivial.

\end{proposition}

\begin{proof} Let $\widehat{B}$ be the mapping cylinder of $p$ and
 $\xymatrix{E\ar@{ >->}[r]^{\hat p} &\widehat{B}\ar[r]^h_{\simeq }&B}$
be the associated factorization of $p$ in a cofibration followed
by a homotopy equivalence. We have $\wsecat(p)=\wsecat(\hat p)$,
$C_p=\widehat{B}/E$ and the identification map
$\rho:\widehat{B}\to C_p$ is naturally homotopic to $jh$.
Applying Proposition \ref{hp-sectional-fatwedge} to the diagram
$$\xymatrix{E \ar[r]\ar@{ >->}[d]^{\hat p} &\ast \ar@{ >->}[d]\\
 \widehat{B}\ar[r]^{\rho} &C_p
}$$ we have, for each $n$, the following pushout:
 $$\xymatrix{T^n(\hat p) \ar[r]\ar@{ >->}[d] &T^n(C_p)\ar@{ >->}[d]\\
 \widehat{B}^{n+1}\ar[r]^{\rho^{n+1}} &C_p^{n+1}
}$$ The vertical maps are cofibrations and, by taking the
quotients, we get the commutative diagram:
 $$\xymatrix{
 \widehat{B}^{n+1}\ar[r]^{\rho^{n+1}}\ar[d]^{q_{n+1}^{\hat p}}&C_p^{n+1}\ar[d]^{q_{n+1}^{C_p}}\\
 \widehat{B}^{n+1}/T^n(\hat p)\ar[r]^(.6){\simeq}&C_p^{[n+1]}
}$$ Therefore, $\wsecat(\hat p)\leq n$ if and only if
$q_{n+1}^{C_p}\rho^{n+1}\Delta_{n+1}^{\widehat{B}}\simeq \ast$
and, since $q_{n+1}^{C_p}\rho^{n+1}=
\rho^{[n+1]}q_{n+1}^{\widehat{B}}$, we have $\wsecat(\hat p)\leq
n$ if and only if $\rho^{[n+1]}\bar
\Delta^{\widehat{B}}_{n+1}\simeq \ast$. Now, using the naturality
of the reduced diagonal and the fact that $\rho\simeq jh$ we
obtain $\rho^{[n+1]}\bar \Delta^{\widehat{B}}_{n+1}\simeq
j^{[n+1]}h^{[n+1]}\bar \Delta^{\widehat{B}}_{n+1}\simeq
j^{[n+1]}\bar \Delta^{B}_{n+1}h$. Since $h$ is a homotopy
equivalence, we finally have $\wsecat(\hat p)\leq n$ if and only
if $j^{[n+1]}\bar \Delta^B_{n+1}\simeq \ast$.
\end{proof}

As a direct consequence of this characterization we get

\begin{corollary}\label{wsecat(p)<=wcat(B)}
For any map $p: E\to B$ one has $\wsecat(p)\leq \wcat(B)$.
\end{corollary}

\subsection{Weak sectional category and homotopy pushouts}
Now we establish an important relation between the weak sectional
category and homotopy pushouts. This relation will have interesting
consequences.

\begin{theorem}
{\rm Consider any homotopy commutative diagram
$$
\xymatrix{ {B'} \ar[d]^{\beta } & {A'} \ar[l]_{g'} \ar[d]^{\alpha
}
\ar[r]^{f'} & {C'} \ar[d]^{\gamma } \\
{B} & {A} \ar[l]^g \ar[r]_f & {C} }
$$
\noindent If $\delta :D'\rightarrow D$ is the induced map between
the homotopy pushouts of the corresponding rows of the diagram,
then
$$\mbox{wsecat}(\delta )\leq \mbox{wsecat}(\beta
)+\mbox{wsecat}(\gamma )+1.$$
In particular, if $A'=B'=C'=*$, we have $\mbox{wcat}(D )\leq \mbox{wcat}(B)
+\mbox{wcat}(C )+1.$}
\end{theorem}

\begin{proof}
Using factorizations through a cofibration followed by a homotopy
equivalence, we can suppose without losing generality, that $g$
and $f$ are cofibrations and that $D$ is the honest pushout of $f$
and $g.$ We obtain the fo\-llo\-wing homotopy commutative cube,
where $\overline{f}$ and $\overline{g}$ (the cobase change of $f$
and $g,$ respectively) are also cofibrations:

$$\xymatrix@R=0.45cm@C=1.02cm@!0{
  & {A'} \ar[rrr]^{f'} \ar'[dd][dddd]^{\,\alpha } \ar[ddl]_{g'}
      &  &  & {C'} \ar[dddd]^{\gamma } \ar[ddl]_{\overline{g'}}     \\
                &  &  & \\
  {B'} \ar[rrr]^{\overline{f'}} \ar[dddd]_{\beta }
      &  &  & {D'} \ar@{.>}[dddd]^(.3){\delta} \\
                &  &  &   \\
  & {A} \ar@{ >->}'[rr]^(.6){f}[rrr]  \ar@{ >->}[ddl]_(.35){g} &  &  & {C}
  \ar@{ >->}[ddl]^{\overline{g}}
       \\
                &  &  &\\
  {B} \ar@{ >->}[rrr]^{\overline{f}} &  &  &  {D} \\
}$$ On the other hand, taking homotopy cofibres of the respective
columns, it is easy to obtain a cube in which all the vertical
faces (by assumption also the top face) are strictly commutative
$$\xymatrix@R=0.45cm@C=1.02cm@!0{
  & {A} \ar@{ >->}[rrr]^{f} \ar'[dd][dddd]^{j_{\alpha }} \ar@{ >->}[ddl]_{g}
      &  &  & {C} \ar[dddd]^{j_{\gamma }} \ar@{ >->}[ddl]_{\overline{g}}     \\
                &  &  & \\
  {B} \ar@{ >->}[rrr]^{\overline{f}} \ar[dddd]_{j_{\beta }}
      &  &  & {D} \ar[dddd]^(.3){j_{\delta}} \\
                &  &  &   \\
  & {C_{\alpha }} \ar'[rr]^(.6){h_{\alpha \gamma }}[rrr]  \ar[ddl]_(.35){h_{\alpha \beta}}
  &  &  & {C_{\gamma }}
  \ar[ddl]^{h_{\gamma \delta }}
       \\
                &  &  & \\
  {C_{\beta }} \ar[rrr]^{h_{\beta \delta }} &  &  &  {C_{\delta }} \\
}$$ Now suppose that $\wsecat (\beta)\leq m,$ $\wsecat (\gamma
)\leq n$ and take nullhomotopies $H_{\beta }:j_{\beta
}^{[m+1]}\bar \Delta _{m+1}^B\simeq *,$ $H_{\gamma}:j_{\gamma
}^{[n+1]}\bar \Delta _{n+1}^C\simeq *.$ From the homotopy
extension property of the cofibration
$\overline{f}:B\rightarrowtail D$ we obtain a commutative diagram
$$\xymatrix{
  {B} \ar@{ >->}[d]_{\overline{f}} \ar[rrr]^{i_0}
 & & & {B\times I} \ar[d]_{h_{\beta \delta }^{[m+1]}H_{\beta }}
               \ar@/^2pc/[ddrr]^ {\overline{f}\times id} \\
  {D} \ar[rrr]_{j_{\delta }^{[m+1]}\bar \Delta _{m+1}^D}
  \ar@/_2pc/[drrrrr]_{i_0}
                &  & & {C_{\delta }^{[m+1]}}     \\
                &       &        &  & & {D\times I}
                \ar@{.>}[ull]|-{H_1}}$$
Similarly, we also obtain a homotopy $H_2:D\times I\rightarrow
C_{\delta }^{[n+1]}$ such that $H_2i_0=j_{\delta }^{[n+1]}\bar
\Delta _{n+1}^D$ and $H_2(\overline{g}\times id)=h_{\gamma \delta
}^{[n+1]}H_{\gamma }.$ Then it is easy to check that the composite
$D\times I\stackrel{(H_1,H_2)}{\longrightarrow }C_{\delta
}^{[m+1]}\times C_{\delta }^{[n+1]}{\longrightarrow }C_{\delta
}^{[m+n+2]}$ defines a homotopy between $j_{\delta }^{[m+n+2]}\bar
\Delta _{m+n+2}^D$ and the trivial map.

\end{proof}


\begin{corollary}\label{wsecat(composition)}
{\rm Consider $E\stackrel{p}{\rightarrow
}B\stackrel{j}{\rightarrow }C_p$ a cofibre sequence. If
$f:X\rightarrow B$ is any map then
$$\mbox{wsecat}(jf)\leq \mbox{wsecat}(f)+1$$}
\end{corollary}

\begin{proof}
Apply the above result to the following commutative diagram
$$
\xymatrix{ {*} \ar[d] & {*} \ar[l] \ar[d]
\ar[r] & {X} \ar[d]^f \\
{*} & {E} \ar[l] \ar[r]_p & {B} }
$$
Observe that in this case, the induced map between the homotopy
pushouts verifies $\delta \simeq jf.$
\end{proof}

\begin{lemma}\label{wsecat(trivialmap)}
{\rm Let $f:X\rightarrow Y$ be any map such that $f\simeq *.$ Then
$$\mbox{wsecat}(f)=\mbox{wcat}(Y).$$ }
\end{lemma}

\begin{proof}
By Corollary \ref{wsecat(p)<=wcat(B)}, it suffices to prove that
$\wsecat(f)\geq \wcat(Y)$. Since $f\simeq *$ there is a
factorization $f=Fk,$ where $F:CX\rightarrow Y$ and
$k:X\rightarrow CX$ is the natural inclusion on the cone of $X$.
Taking the following pushout,
$$\xymatrix{
  {X} \ar[d]_{k} \ar[r]^{f}
                & {Y} \ar[d]^{\overline{k}} \ar@/^/[ddr]^{id}  \\
  {CX} \ar[r]_{\overline{f}} \ar@/_/[drr]_{F}
                & {C_f} \ar@{.>}[dr]|-{\varphi }            \\
                &               & {Y}             }$$
we get a map $\varphi :C_f\rightarrow Y$ which is a retraction of
the inclusion $\overline{k}:Y\rightarrow C_f$. We thus have
$\varphi ^{[n+1]}(\overline{k})^{[n+1]}=id$. Therefore, if the map
$\xymatrix{ Y\ar[r]^{\bar \Delta^Y_{n+1}}&
Y^{[n+1]}\ar[r]^{\overline{k}^{[n+1]}}&C_f^{[n+1]}}$ is
homotopically trivial, then so is the reduced diagonal $\bar
\Delta^Y_{n+1}:Y\to Y^{[n+1]}$. By Proposition
\ref{characterization_wsecat}, we obtain
 the expected inequality.
\end{proof}

\begin{corollary}
{\rm Consider $E\stackrel{p}{\rightarrow
}B\stackrel{j}{\rightarrow }C_p$ a cofibre sequence. Then
$$\mbox{wcat}(C_p)\leq \mbox{wsecat}(p)+1$$ }
\end{corollary}

\begin{proof}
Apply Corollary \ref{wsecat(composition)} and Lemma
\ref{wsecat(trivialmap)} when $f=p.$.
\end{proof}

\noindent Observe that this inequality is an improvement of the
classical inequality $\wcat(C_p)\leq \wcat(B)+1$.

\subsection{Some remarkable inequalities}

In the two previous sections we have seen that, for a map $p: E\to
B$ with homotopy cofibre $C_p$, we have $\wsecat(p)\leq \wcat(B)$
and $\wsecat(p)\geq \wcat(C_p)-1$. Actually these two inequalities
fit in a set of inequalities we collect in the following
statement:
\begin{theorem}\label{inequalities}
If $p: E\to B$ is a map and $C_p$ its homotopy cofibre, then
\begin{enumerate}
\item[(a)] $\wsecat(p)\leq \wcat(B)$
\item[(b)] $\wsecat(p)\leq \wcat(C_p)$
\item[(c)] $\wsecat(p)\geq \wcat(C_p)-1$
\item[(d)] $\wsecat(p)\geq \nil \ker p^*$ (for any commutative ring $\pi$)
\item[(e)] If the map $p: E\to B$ admits a homotopy retraction, then
$$\wsecat(p)=\wcat(C_p)\hspace{5pt}\mbox{and}\hspace{5pt}\nil \ker p^*=\cuplength(C_p).$$
\end{enumerate}

\end{theorem}

\noindent Observe that the inequalities (a), (b), (c) and (d) can
be summarized in the following way:
$$\max(\wcat(C_p)-1,\nil \ker p^*) \leq \wsecat(p)\leq \min(\wcat(C_p),\wcat(B))$$
After the proof of the theorem, we will see, through examples, that all the
inequalities (a), (b), (c) and (d) can be strict. In particular, this shows
that $\wsecat(p)$ is a better lower bound for $\secat(p)$ than the classical
cohomological lower bound $\nil \ker p^*$. \\

\begin{proof} We just have to prove (b), (d) and (e). We use the characterization of
$\wsecat$ established in Proposition \ref{characterization_wsecat}. We first prove (b).
By naturality of the diagonal the following diagram is commutative:
$$\xymatrix{ B\ar[rr]^{\bar \Delta_{n+1}^B}\ar[d]_j & & B^{[n+1]}\ar[d]^{j^{[n+1]}}\\
C_p\ar[rr]^{\bar \Delta_{n+1}^{C_p}} & & C_p^{[n+1]}}$$
Therefore if $\wcat(C_p)\leq n,$ then we have $\wsecat(p)\leq n$.\\
We now prove (d). Assume that $\mbox{wsecat}(p)\leq n$ and take
$x_1,...,x_{n+1}\in \ker(p^*)\subset \tilde{H}^*(B)$  where
$\tilde H^*$ stands for the reduced cohomology with coefficients
in a commutative ring $\pi$. From the cofibre sequence
$E\stackrel{p}{\longrightarrow }B\stackrel{j}{\longrightarrow}C_p$
we have an exact sequence in reduced cohomology
$$\tilde{H}^*(C_p)\stackrel{j^*}{\longrightarrow
}\tilde{H}^*(B)\stackrel{p^*}{\longrightarrow}\tilde{H}^*(E)$$
\noindent so we can consider $y_1,...,y_{n+1}\in \tilde{H}^*(C_p)$
such that $j^*(y_i)=x_i,$ $i=1,...,n+1.$ Now take, for each $i,$ a
representative $f_i:C_p\rightarrow K_i$ of $y_i$ (where $K_i=K(\pi
,m_i)$ is the Eilenberg-MacLane space of type $(\pi ,m_i),$ for
certain $m_i$) and the following commutative diagram:
$$
\xymatrix{ {B} \ar[rrdd]_{\bar \Delta _{n+1}^{B}} \ar[r]^{j} &
{C_p}\ar[rrdd]^{\bar \Delta _{n+1}^{C_p}}
 \ar[rr]^{\Delta _{n+1}^{C_p}} & & {C_p^{n+1}}
\ar[dd]^{q_{n+1}^{C_p}} \ar[rr]^(.4){f_1\times ...\times f_{n+1}}
& &
{K_1\times ...\times K_{n+1}} \ar[dd]^q \\
&&&&&\\
 & & B^{[n+1]} \ar[r]_{j^{[n+1]}}&{C_p^{[n+1]}}
\ar[rr]^(.4){f_1\wedge ...\wedge f_{n+1}} & & {K_1\wedge ...\wedge
K_{n+1}} }
$$
Then $x_1\cup ...\cup x_{n+1}=j^*(y_1\cup ...\cup y_{n+1})$ has as
a representative the composition of $q(f_1\times ...\times
f_{n+1})\Delta _{n+1}^{C_p}j$ with certain map $K_1\wedge
...\wedge K_{n+1}\rightarrow K(\pi, m_1+...+m_{n+1})$ (see, for
instance, \cite{A-G-P}). But the first map is already homotopy
trivial since $j^{[n+1]}\bar \Delta _{n+1}^B\simeq *.$ This
completes the proof of (d). We also observe that if $y_1\cup
...\cup y_{n+1}=0$, then $x_1\cup ...\cup x_{n+1}=j^*(y_1\cup
...\cup y_{n+1})=0$.
In other words, $\nil \ker p^*\leq\cuplength(C_p)$.\\

Finally we prove (e). Suppose that $\wsecat(p)\leq n$ and let $H$ be a
homotopy between $j^{[n+1]}\bar \Delta_{n+1}^B=\bar \Delta_{n+1}^{C_p}j$
and the trivial map. This homotopy together with the choice of a homotopy
equivalence between $C_j$ and $\Sigma E$ determines a map
$\phi_H:\Sigma E \to C_p^{[n+1]}$ such that the lower triangle in the
following diagram is homotopy commutative: \\
$$\xymatrix{ B \ar[rr]^{\bar \Delta_{n+1}^B} \ar[d]_j & & B^{[n+1]}\ar[d]^{j^{[n+1]}}\\
C_p\ar[rr]^{\bar \Delta_{n+1}^{C_p}}\ar[d]_{\delta} & & C_p^{[n+1]}\\
\Sigma E \ar[rru]_{\phi_H} & & }$$ Here $\delta$ is the
identification map $C_p\to C_p/B=\Sigma E$. It follows directly
from this diagram that if $\delta\simeq *$ or $\phi_H\simeq *,$
then $\wcat(C_p)\leq n$ and we can conclude that
$\wsecat(p)=\wcat(C_p)$. In particular, if $p$ has a homotopy
retraction $r$, then $\Sigma r$ is a homotopy retraction of
$\Sigma p$ and by considering the Barratt-Puppe sequence
$$\xymatrix{
E\ar[r]^p&B\ar[r]^j & C_p \ar[r]^{\delta}& \Sigma E \ar[r]^{\Sigma
p} & \Sigma B \ar[r]&... }$$ we see that $\delta\simeq \Sigma r
\Sigma p\, \delta \simeq \ast$. This proves the first part of the statement.
For the second part, we observe that if $p$ has a homotopy retraction,
then $p^*$ is surjective and $\ker p^*=\tilde H^*(C_ p)$.
Therefore $\nil \ker p^*=\nil \tilde H^*(C_ p)=\cuplength(C_p)$.\\
\end{proof}

The following examples show that the inequalities (a), (b), (c) and (d),
in this order, can be strict.\\

\noindent \textbf{Examples}

\begin{enumerate}
\item If $B$ is a space with $\wcat(B)>0$ then the weak sectional category of the
projection $p:A\times B\to B$, where $A$ is any space, satisfies
$\wsecat(p)=0<\wcat(B)$.

\item Let $p$ be the Hopf map $S^3\to S^2$. We have $\wsecat(p)\leq \wcat(S^2)=1$
but $\wcat(C_p)=\wcat(\CP^2)=2$.

\item If $p$ admits a homotopy retraction, then $\wsecat(p)=\wcat(C_p)>\wcat(C_p)-1$.
A concrete example is given by taking for $p$ an inclusion $B\to B\vee A$.

\item Recall that the symplectic group Sp(2) admits a cellular decomposition of the form
$S^3\cup_{\alpha}e^7\cup_{\beta}e^{10}$. It has been proved by
Berstein and Hilton \cite{B-H} that
$\wcat(S^3\cup_{\alpha}e^7)=\cat(S^3\cup_{\alpha}e^7)=2$ and by
Schweitzer \cite{Schwe} that
$\wcat(\mbox{Sp(2)})=\cat(\mbox{Sp(2)})=3$. Take $p=\beta:S^9\to
S^3\cup_{\alpha}e^7$. Since $\wcat(S^3\cup_{\alpha}e^7)=2$, we
have $\wsecat(p)\leq 2$. On the other hand, since
$C_p=\mbox{Sp(2)}$, we have $\wsecat(p)\geq \wcat(C_p)-1=2$.
Therefore $\wsecat(p)=2$. But $\ker(p^*)=\tilde
H^*(S^3\cup_{\alpha}e^7)$
and $\nil(\ker p^*)=1$.\\

\end{enumerate}


Our last result for this section establishes the equality between
the weak sectional category and the sectional category under some
restrictions on the dimension and connectivity. This kind of
results is classical in the theory of Lusternik-Schnirelmann
category (see, for instance, \cite{C-L-O-T} [Section 2.8]) and is
based on Blakers-Massey Theorem that, in this context, it is
appropriate to state in the following form:

\begin{lemma}[\textbf{Blakers-Massey Theorem}]
Let $A\stackrel{j}{\rightarrow }B\stackrel{q}{\rightarrow }C$ be a
cofibre sequence where all spaces are simply connected, $A$ is
$(a-1)$-connected and $C$ is $(c-1)$-connected. Let $F$ be the
homotopy fibre of $q,$ $\iota :F\rightarrow B$ the induced map and
$d:A\rightarrow F$ a lifting of $j$ (i.e. $id\simeq j$). Then $d$
is an $(a+c-2)$-equivalence.
\end{lemma}

\begin{theorem}\label{secat-wsecat}
{\rm Let $p:E\rightarrow B$ be any map, where $E$ and $B$ are
(q-1)-connected CW-complexes ($q\geq 2$). If dim(B)=N and either
one of the following conditions is satisfied
\begin{enumerate}
\item[(i)] $N\leq q(\wsecat(p)+2)-2$

\item[(ii)] $N\leq q(\secat(p)+1)-2,$
\end{enumerate}

\noindent then $\secat(p)=\wsecat(p).$ }
\end{theorem}

\begin{proof}
For the proof of (i) suppose that $\wsecat (p)=n$ and consider the
following diagram
$$\xymatrix{
{\widetilde{F}_{n+1}} \ar[r]^{\iota _n} & {B^{n+1}} \ar[r]^{l_n} &
{C_{\kappa _n}} \\
{T^n(p)} \ar[ur]_{\kappa _n} \ar@{.>}[u]^{\widetilde{\kappa }_n}
}$$ where $\widetilde{F}_{n+1}$ the homotopy fiber of $l_n$ and
$\widetilde{\kappa }_n$ is the map induced by the homotopy fiber
property. By the same property we can also consider a lift
$\widetilde{\Delta }_{n+1}^B:B\rightarrow \widetilde{F}_{n+1}$ of
the diagonal map $\Delta _{n+1}^B.$ Taking into account that
$T^n(p)$ is (q-1)-connected and $C_{\kappa _n}$ is
(q(n+1)-1)-connected then, by the Blakers-Massey Theorem,
$\widetilde{\kappa }_n$ is a (q(n+2)-2)-equivalence. So, by the
hypothesis on the dimension of $B,$ we have that
$$(\widetilde{\kappa }_n)_*:[B,T^n(p)]\rightarrow
[B,\widetilde{F}_{n+1}]$$ \noindent is a surjection and therefore
there exists a map $\widehat{\Delta }_{n+1}^B:B\rightarrow T^n(p)$
such that $\widetilde{\kappa }_n\widehat{\Delta }_{n+1}^B \simeq
\widetilde{\Delta }_{n+1}^B.$ Composing $\iota _n$ to both sides
we obtain a factorization $\kappa _n\widehat{\Delta }_{n+1}^B
\simeq \Delta _{n+1}^B;$ that is, $\secat (p)\leq n.$

For (ii) suppose that $\secat(p)=n$ and that $\wsecat(p)\leq n-1.$
By a similar argument to that above one can find a factorization
$\kappa _{n-1}\widehat{\Delta }_{n}^B \simeq \Delta _{n}^B,$ which
is a contradiction.
\end{proof}

\section{Weak topological complexity.}

Studying the motion planning problem by using topological
techniques, M. Farber introduced in \cite{Far} and \cite{Far2} the
topological complexity of any space $X,$ $\mbox{TC}(X),$ as the
sectional category of the evaluation fibration $\pi
_X:X^I\rightarrow X\times X,$ $\pi _X(\alpha )=(\alpha (0),\alpha
(1)).$ Broadly speaking, this invariant measures the discontinuity
of any motion planner in the space. As the topological complexity
of a space $X$ is the sectional category of the fibration $\pi _X$
we may consider our weak version of sectional category to obtain a
\textit{new} lower bound of this homotopy invariant. Namely, we
define the \textit{weak topological complexity} of $X$ as $$\wTC
(X):=\wsecat (\pi _X)$$ Observe that, since $\pi _X$ is the
mapping path fibration associated to the diagonal map $\Delta
_X:X\rightarrow X\times X,$ we have that $\wTC (X)=\wsecat (\Delta
_X).$ We also note that in many cases the $n$-fatwedge of $\Delta
_X$ has the nice description given in Section 2. Namely, for
locally equiconnected spaces (those spaces in which $\Delta _X$ is
a cofibration) $T^n(\Delta _X)$ has, up to homotopy equivalence,
the following explicit expression
$$T^n(\Delta _X)=\{(y_0,...,y_n)\in (X\times X)^{n+1}\, |\,y_i\in \Delta _X(X)
\hspace{3pt} \mbox{for some}\hspace{3pt}i\}$$ Under the same
condition on $X,$ the homotopy cofiber of $\Delta _X$, $C_{\Delta
_X}$, is
 homotopically equivalent to the quotient space
$$(X\times X)/\Delta _X(X).$$

\noindent Note that the class of locally equiconnected spaces is
not very restrictive. For instance, the CW-complexes and the
metrizable
topological manifolds fit on such class of spaces. \\

The next two results are just a specialization of Theorem
\ref{inequalities} and Theorem \ref{secat-wsecat}, respectively,
and therefore their proofs are omitted. Recall that, if $\tilde
H^*$ stands for the reduced cohomology with coefficients in a
field \textbf{K}, then M. Farber \cite{Far} proved that $\nil \ker
(\Delta _X)^*=\nil \ker \hspace{3pt}\cup ,$ where $\cup :\tilde
H^*(X)\otimes \tilde H^*(X)\rightarrow \tilde H^*(X)$ denotes the
usual cup-product.

\begin{theorem}
Let $X$ be any space. Then
\begin{enumerate}
\item[(a)] $\wTC(X)\leq \wcat(X\times X)$
\item[(b)] $\wTC(X)\geq \nil \ker \hspace{3pt}\cup$
\item[(c)] $\wTC(X)=\wcat(C_{\Delta _X})$ (and $\nil \ker \hspace{3pt}\cup
=\cuplength(C_{\Delta _X})$).
\end{enumerate}
\end{theorem}

\begin{theorem}
{\rm Let $X$ be any (q-1)-connected CW-complex ($q\geq 2$). If
dim(X)=N and either one of the following conditions is satisfied
\begin{enumerate}
\item[(i)] $N\leq \frac{q(\wTC(X)+2)}{2}-1$

\item[(ii)] $N\leq \frac{q(\TC(X)+1)}{2}-1,$
\end{enumerate}

\noindent then $\TC(X)=\wTC(X).$ }
\end{theorem}

Finally we give two concrete computations. The first one
(Proposition \ref{cofibre_Diag_Sphere} below) consists of the
explicit determination of the homotopy cofibre of the diagonal map
$\Delta_{S^n}: S^n\to S^n\times S^n$ of a sphere. Using the
results of \cite{B-H} together with the classical results on the
Hopf invariant of the Whitehead product $[\iota_n,\iota_n]$ (where
$\iota_n$ denotes the homotopy class of the identity of $S^n$) it
is then possible to deduce that
$\wTC(S^n)(=\wcat(C_{\Delta_{S^n}}))$ is $1$ if $n$ is odd and is
$2$ if $n$ is even. The previous theorem permits to recover the
result by Farber that the topological complexity of an odd
dimensional sphere is $1$ while that of an even dimensional sphere
is $2$.

\begin{proposition}\label{cofibre_Diag_Sphere}
The homotopy cofibre of $\Delta_{S^n}: S^n\to S^n\times S^n$ is
homotopy  equivalent to $S^n\cup_{[\iota_n,\iota_n]}e^{2n}$.
\end{proposition}

\begin{proof}
Let denote by $\nu:S^n\to S^n\vee S^n$ the pinch map, by
$j:S^n\vee S^n\to S^n\times S^n$ the inclusion and by
$\tilde{\nabla}$ the map $\xymatrix{S^n\vee S^n
\ar[r]^-{(id,-id)}& S^n}$. It is sufficient to establish the
result for the homotopy cofibre of $j\nu$ since this composite is
homotopic to the diagonal map $\Delta_{S^n}$. It is not hard to
check that the following homotopy commutative diagram is a
homotopy pushout
$$\xymatrix{S^n\ar[r]^{\nu}\ar[d]&S^n\vee S^n \ar[d]^-{\tilde{\nabla}}\\
\ast\ar[r]& S^n}$$ We can thus consider the following diagram in
which $w$ is the attaching  map of the top cell of $S^n\times S^n$
and the two lower squares are pushouts and homotopy pushouts
$$\xymatrix{& S^n\ar[d]^{\nu}\ar[r] &\ast\ar[d]\\
S^{2n-1}\ar[r]^{w}\ar@{ >->}[d] &
S^n\vee S^n\ar[r]^{\tilde{\nabla}}\ar@{ >->}[d]_j& S^n\ar@{ >->}[d]\\
CS^{2n-1}\ar[r] & S^n\times S^n\ar[r]& Z}$$ By composition we see
that, on the one hand, $Z$ is homotopy equivalent to the homotopy
cofibre of $j\nu$ and, on the other hand, $Z$ is the homotopy
cofibre of $\tilde{\nabla}w$ whose homotopy class is the Whitehead
product $[\iota_n,-\iota_n]= -[\iota_n,\iota_n]$. We then have
$Z\simeq S^n\cup_{-[\iota_n,\iota_n]}e^{2n}\simeq
S^n\cup_{[\iota_n,\iota_n]}e^{2n},$ where the last equivalence is
induced by  $-id:S^{2n-1}\to S^{2n-1}$.
\end{proof}

Our second concrete computation has the objective to show, through
an example, that $\wTC$ is, in general, a better lower bound for
the topological complexity than $\nil \ker \hspace{3pt}\cup$.
Before giving this example, we first prove a useful general result
about weak category.

\begin{proposition}\label{growth-of-wcat}
Let $f:A\to B$ be a map between $(q-1)$-connected CW-complexes
($q\geq 1$). If $f$ is an $r$-equivalence, $\wcat(A)\geq k$ and
$\dim(A)\leq r+q(k-1)-1$ then $\wcat(B)\geq \wcat(A)$.
\end{proposition}

\begin{proof}
Since $\wcat(A)\geq k$ the $k$-reduced diagonal $\bar
\Delta_k^A:A\to A^{[k]}$ is not homotopically trivial. By the
naturality of this diagonal we have the following commutative
diagram:
$$\xymatrix{
A \ar[r]^{\bar \Delta_k^A} \ar[d]_f & A^{[k]}\ar[d]^{f^{[k]}}\\
B \ar[r]_{\bar \Delta_k^B} & B^{[k]} }$$ Since $f$ is an
$r$-equivalence and $A$ and $B$ are $(q-1)$-connected CW-complexes
we obtain that $f^{[k]}$ is an $r+q(k-1)$ equivalence. The
assumption on the dimension of $A$ permits thus to assert that
$$(f^{[k]})_*:[A,A^{[k]}]\to [A, B^{[k]}]
$$
is injective. Since $\bar \Delta_k^A$ is not homotopically trivial
we conclude that $\bar \Delta_k^B\circ f=f^{[k]}\circ\bar
\Delta_k^A$ and $\bar \Delta_k^B$ are not homotopically trivial.
Therefore $\wcat(B)\geq k$.
\end{proof}

\begin{proposition}
Let $X:=S^3\cup_{\alpha} e^7$ where $\alpha:S^6\to S^3$ is the
Blakers-Massey element (that is, $X$ is the $7$-skeleton of
$\mbox{Sp(2)}$). For this space, we have $\nil \ker
\hspace{3pt}\cup=2$ and $\wTC(X)\geq 3$.
\end{proposition}

\begin{proof}
Since $H^*(X)=H^*(S^3\vee S^7)$, it is easy to check that $\nil
\ker \hspace{3pt} \cup=2$.  In order to prove that
$\wTC(X)=\wcat(C_{\Delta_X})\geq 3$ we proceed in three steps. Let
$C$ be the homotopy cofibre of the composition
$$\xymatrix{(incl\times id)\Delta_{S^3}:S^3
\ar[r]^-{\Delta_{S^3}}& S^3\times S^3 \ar[r]^{incl\times
id}&X\times S^3}.$$ The first step aims to see that it suffices to
show that $\wcat(C)\geq 3$. The second step shows that $C$ fits in
a special diagram and the third step gives the proof that
$\wcat(C)\geq 3$.

\textit{Step 1. } Let $f:C\to C_{\Delta_X}$ be the map induced by
the horizontal maps of the following square when we take the
homotopy cofibres of the vertical maps.
$$\xymatrix{
S^3 \ar[r]^{incl} \ar[d]_{(incl\times id)\Delta_{S^3}} & X \ar[d]^{\Delta_X} \\
X\times S^3 \ar[d] \ar[r]_{id\times incl} & X\times X \ar[d]\\
C \ar@{.>}[r]_f& C_{\Delta_X} }$$ Taking into account that the map
$X\times X\to C_{\Delta_X}$ induces a surjective morphism in
homology (because $\Delta_X$ admits a retraction), it is easy to
see that $H_i(f)$ is an isomorphism for $i\leq 5$ and is
surjective for $i=6$. Since $C$ and $C_{\Delta_X}$ are
$2$-connected (and thus $1$-connected) we can conclude that $f$ is
a $6$-equivalence. Since $\dim(C)=10$, we apply Proposition
\ref{growth-of-wcat} with $q=3$
and $r=6$ to obtain that if $\wcat(C)\geq 3,$ then $\wcat(C_{\Delta_X})\geq 3$.\\

\textit{Step 2. }Consider now the map $\xymatrix{S^3\times S^3
\ar[r]^{\mu'}& S^3}$  given by $\mu'(x,y)=x\cdot y^{-1}$. Since
$\mu'\Delta_{S^3}$ is trivial, $\mu'$ factors through
$C_{\Delta_{S^3}}$ and we have the following commutative diagram
$$\xymatrix{
S^3\times S^3 \ar[r]^{\varphi} \ar[rd]_{\mu'} & C_{\Delta_{S^3}}\ar[d]^{\rho}\\
& S^3 }$$ \noindent Let $i_1:S^3\to S^3\times S^3$ be the
inclusion of the first factor. We will see that there exists a
commutative diagram of the form
$$\xymatrix{
 C_{\Delta_{S^3}}\cup_{\varphi i_1 \alpha} e^7\ar@{ >->}[r] \ar[d] & C \ar[d]\\
X \ar@{ >->}[r] & {\mbox{Sp(2)}} }$$ which is a pushout and in
which the map $X\to \mbox{Sp(2)}$ is the inclusion and the  map
$C\to \mbox{Sp(2)}$ is a $5$-equivalence. We first prove that the
pushout of the cofibration $\xymatrix{S^3\times S^3 \ar@{
>->}[r]^{incl\times id}& X\times S^3}$ and the map
$\xymatrix{S^3\times S^3 \ar[r]^{\mu'}& S^3}$ is the group
$\mbox{Sp(2)}$. Since $\mbox{Sp(2)}$ is the total space of the
$S^3$-principal bundle over $S^7$ classified by the adjoint of
$\alpha: S^6\to S^3=\Omega B S^3$, we have the following pushout
diagram
$$\xymatrix{S^6\times S^3 \ar[d]_{\chi}\ar@{ >->}[r]^{incl \times id}& CS^6\times S^3\ar[d]\\
S^3 \ar@{ >->}[r] & {\mbox{Sp(2)}} }$$ where $\chi$ is the
composition $\xymatrix{S^6\times S^3 \ar[r]^{\alpha \times id}&
S^3\times S^3 \ar[r]^{mult}&S^3}$ and the map $S^3\to
\mbox{Sp(2)}$ is the inclusion. This diagram admits the following
decomposition in which the two upper squares are pushouts:
$$\xymatrix{S^6\times S^3 \ar[d]_{\alpha \times id}\ar@{ >->}[r]^{incl \times id}&
CS^6\times S^3\ar[d]\\
S^3\times S^3 \ar[d]_{id \times inv}\ar@{ >->}[r]^{incl \times id} &
X\times S^3\ar[d]^{id \times inv} \\
S^3\times S^3 \ar[d]_{\mu'}\ar@{ >->}[r]^{incl \times id}& X\times S^3\ar[d] \\
S^3 \ar@{ >->}[r] & {\mbox{Sp(2)}}}$$ We then conclude that the
bottom square is also a pushout diagram. We now decompose  $\mu'$
in $\rho \varphi$ and consider the pushout of $\varphi$ and the
cofibration $incl\times id:S^3\times S^3\to X\times S^3$. This
pushout is exactly the homotopy cofibre of $(incl\times
id)\Delta_{S^3}$, that is $C$. In this way the bottom square of
the previous diagram admits the following decomposition into two
pushouts:
$$\xymatrix{
S^3\times S^3 \ar[d]_{\varphi}\ar@{ >->}[r]^{incl \times id}& X\times S^3\ar[d]_{\hat{\varphi}} \\
C_{\Delta_{S^3}} \ar[d]_{\rho} \ar@{ >->}[r] & C \ar[d]_{\hat{\rho}}\\
S^3 \ar@{ >->}[r] & {\mbox{Sp(2)}}}$$ Now we decompose the
cofibration $incl\times id:S^3\times S^3\to X\times S^3$ in the
two following cofibrations
$$\xymatrix{
S^3\times S^3 \ar@{ >->}[r]& S^3\times S^3 \bigcup_{S^3\vee S^3}
X\vee S^3\ar@{ >->}[r]&X\times S^3 }$$ and observe that the
indermediate space is exactly the homotopy cofibre of
$i_1\alpha:S^6\to S^3\times S^3$. By constructing the pushout of
the cofibration $S^3\times S^3\to S^3\times S^3 \bigcup_{S^3\vee
S^3} X\vee S^3$ with first $\varphi$ and secondly $\rho\varphi$ we
then obtain the homotopy cofibres of $\varphi i_1 \alpha$ and
$\rho\varphi i_1 \alpha=\alpha$ (note that
 $\rho\varphi i_1=\mu' i_1=id$) and the following commutative diagram
 in which each square is a pushout.

$$\xymatrix{S^3\times S^3 \ar[d]_{\varphi} \ar@{ >->}[r] &
S^3\times S^3 \bigcup_{S^3\vee S^3} X\vee S^3\ar@{ >->}[r]\ar[d]&X\times S^3\ar[d]\\
C_{\Delta_{S^3}} \ar[d]_{\rho} \ar@{ >->}[r] &
C_{\Delta_{S^3}}\cup_{\varphi i_1 \alpha} e^7 \ar[d]_{\tilde{\rho}}
\ar@{ >->}[r] &C \ar[d]_{\hat{\rho}}\\
S^3 \ar@{ >->}[r] & X \ar@{ >->}[r] & {\mbox{Sp(2)}}}$$ The right
bottom pushout is the expected one. Observe that the inclusion
$S^3\times S^3 \bigcup_{S^3\vee S^3} X\vee S^3\hookrightarrow
X\times S^3$ is a $9$-equivalence. Thus so is the map $X\to
\mbox{Sp(2)}$. In other words this map is the inclusion. On the
other hand the map $\hat{\rho}$ has the same connectivity as
$\rho$. In order to justify that $\rho$ is a $5$-equivalence we
just recall that $C_{\Delta_{S^3}}=S^3\cup_{[\iota_3,-\iota_3]}e^6
\simeq S^3\vee S^6$ because the Whitehead products are trivial in
$S^3$ and that $\rho$ satisfies
$\rho\varphi i_1=id$.\\

\textit{Step 3. } We finally prove that $\wcat(C)\geq 3$. The
homotopy cofibre  of the inclusion $X\to \mbox{Sp(2)}$ is
homotopically equivalent to $S^{10}$. Therefore, from the pushout
of the previous step, we have the following homotopy commutative
diagram.
$$\xymatrix{
 C \ar[d]_{\hat{\rho}} \ar[r]^{\partial} & S^{10} \ar@{=}[d]\\
{\mbox{Sp(2)}}\ar[r]_{\delta}& S^{10} }$$ On the other hand, by
naturality, the following diagram is commutative
$$\xymatrix{
 C \ar[d]_{\hat{\rho}} \ar[r]^-{\bar{\Delta}_3} & C\wedge C\wedge C
 \ar[d]^{\hat{\rho}\wedge \hat{\rho}\wedge\hat{\rho}}\\
{\mbox{Sp(2)}} \ar[r]^-{\bar{\Delta}_3}& {\mbox{Sp(2)}\wedge
\mbox{Sp(2)}\wedge \mbox{Sp(2)}} }$$ Since $\wcat(X)\leq 2$, the
$3$rd reduced diagonal of $Sp(2)$ factors through $\delta$ and we
have a homotopy commutative diagram
$$\xymatrix{
{\mbox{Sp(2)}} \ar[rd]_{\delta}\ar[rr]^-{\bar{\Delta}_3}& &
{\mbox{Sp(2)}\wedge \mbox{Sp(2)}\wedge \mbox{Sp(2)}}\\
& S^{10} \ar[ru]_h }$$ As $\wcat(\mbox{Sp(2)})>2$ the map $h$ is
not homotopically trivial (actually $h$ is known to be the
composite $S^{10}\stackrel{\Sigma^7\eta}{\to}S^9\hookrightarrow
\mbox{Sp(2)}\wedge \mbox{Sp(2)}\wedge \mbox{Sp(2)}$, see
\cite[Rem. 6.50]{C-L-O-T}). Now, taking into account that
${\hat{\rho}}$ is a $5$-equivalence and that $C$ and
$\mbox{Sp(2)}$ are $2$-connected we can check that
$\hat{\rho}\wedge \hat{\rho}\wedge\hat{\rho}$ is a
$11$-equivalence. Therefore $h$ lifts to a map $m:S^{10} \to
C\wedge C\wedge C$ such that $(\hat{\rho}\wedge
\hat{\rho}\wedge\hat{\rho})m\simeq h$. We then have the following
diagram in which each face, except the top triangle, is homotopy
commutative.
$$\xymatrix{
 C \ar[rd]_{\partial}\ar[dd]_{\hat{\rho}} \ar[rr]^-{\bar{\Delta}_3}
 & & C\wedge C\wedge C \ar[dd]^{\hat{\rho}\wedge \hat{\rho}\wedge\hat{\rho}}\\
& S^{10} \ar@{.>}[ru]_m \ar@{=}[dd]\\
{\mbox{Sp(2)}} \ar[rd]_{\delta} \ar'[r][rr]^(.2){\bar{\Delta}_3} & &
{\mbox{Sp(2)}\wedge \mbox{Sp(2)}\wedge \mbox{Sp(2)}}\\
& S^{10} \ar[ru]_h }$$ Since $\dim(C)=10$, we deduce from the fact
that $\hat{\rho}\wedge \hat{\rho}\wedge\hat{\rho}$ is a
$11$-equivalence that the top triangle is actually homotopy
commutative. Suppose now that $\wcat(C)\leq 2$. This implies that
the composite $m\partial$ is homotopically trivial and, therefore,
that the map $m$ factors through the homotopy cofibre of
$\partial$. Since this homotopy cofibre is homotopy equivalent to
$\Sigma (C_{\Delta_{S^3}}\cup_{\varphi i_1 \alpha} e^7)$ we then
have a homotopy commutative diagram:
$$\xymatrix{
S^{10} \ar[rr]^m \ar[rd] && C\wedge C \wedge C\\
& \Sigma (C_{\Delta_{S^3}}\cup_{\varphi i_1 \alpha} e^7)\ar[ru]&
}$$ But $C\wedge C \wedge C$ is $8$-connected and $\dim(\Sigma
(C_{\Delta_{S^3}}\cup_{\varphi i_1 \alpha} e^7))=8$. Therefore $m$
should be homotopically trivial which contradicts the fact that
$h$ is not homotopically trivial. We then conclude that
$\wcat(C)\geq 3$.\end{proof}


\begin{thebibliography}{99}

\bibitem{A-G-P} {M. Aguilar, S. Gilter \and C. Prieto.} \textit{Algebraic
Topology from a Homotopical Viewpoint}. Springer, 2002

\bibitem{B-H} {I. Berstein \and P.J. Hilton}. Category and
generalized Hopf invariants. \textit{Illinois J. Math.} 4
(1960),437-451

\bibitem{B-K} { A.K. Bousfield \and D.M. Kan.}
\textit{Homotopy limits, completions and localizations}.
Lecture Notes in Mathematics \textbf{304}, 1972.

\bibitem{C-L-O-T} O. Cornea, G. Lupton, J. Oprea and D. Tanré.
\textit{Lusternik-Schnirelmann category}. Math. Surveys and
Monographs, vol. 103, AMS, 2003.

\bibitem{D2} {J.P. Doeraene.} LS-category in a model category. \textit{J.
Pure App. Algebra} 84 (1993), 215-261.

\bibitem{D} { J.P. Doeraene.} Homotopy pullbacks, homotopy pushouts and joins.
\textit{Bull. Belg. Math. Soc.} 5(1) (1998), 15-37.

\bibitem{Far} {M. Farber.} Topological complexity of motion
planning. \textit{Discrete Comput. Geom.} 29 (2003), 211-221

\bibitem{Far2} {M. Farber.} Instabilities of robot motion.
\textit{Topology Appl.} 140 (2004), 245-266


\bibitem{F} {A. Fassò Velenik.} \textit{Relative homotopy invariants of
the type of the Lusternik-Schnirelmann category}. Eingereichte
Dissertation (Ph.D. Thesis), Freie Universität Berlin, 2002.

\bibitem{F-G-K-V} {L. Fernandez Suarez, P. Ghienne, T. Kahl \and L.
Vandembroucq.} Joins of DGA modules and sectional category.
\textit{Algebraic \& Geometric Topology} 6 (2006), 119-144

\bibitem{J} {I.M. James.}
On category in the sense of Lusternik-Schnirelmann.
\textit{Topology} 17 (1978), 331-348

\bibitem{M} { M. Mather.} Pull-backs in Homotopy Theory.
\textit{Can. J. Math.} 28(2) (1976), 225-263.

\bibitem{Sch} {A. Schwarz.} \textit{The genus of a fiber space},
A.M.S. Transl. 55(1966), 49-140.

\bibitem{Schwe} {P.A. Schweitzer.} Secondary cohomology operations induced
by the diagonal mapping. \textit{Topology} 3 (1965), 337-355.

\bibitem{Sm} {S. Smale}. On the topology of algorithms I. \textit{J.
Complexity} 3 (1987), 81-89.




\end{thebibliography}
\end{document}